\documentclass[12pt,a4paper]{article}
\usepackage[T2A]{fontenc}
\usepackage[cp1251]{inputenc}
\usepackage{amsmath,amssymb,amsthm,amsfonts,amscd}
\usepackage[matrix,arrow,curve]{xy}

\begin{document}

\title{On properties of skew-framed immersions cobordism groups 
\footnote{Supported by RFBR Grant No. 15-01-06302}
}
\author{P.M.Akhmet'ev
\footnote{National Research University
``Higher School of Economics'';
IZMIRAN, Kaluzhskoe Sh. 4, Troitsk, Moscow 108840, Russia}, 
O.D.Frolkina
\footnote{Chair of General Topology and Geometry, Faculty of Mechanics and Mathematics, M.V. Lomonosov Moscow State University,
Leninskie Gory 1, GSP-1, Moscow 119991, Russia}
}

\sloppy \theoremstyle{plain}
\newtheorem{theorem}{Theorem}
\newtheorem*{main*}{Main Theorem}
\newtheorem*{theorem*}{Theorem}
\newtheorem{lemma}{Lemma}
\newtheorem{proposition}{Statement}
\newtheorem*{proposition*}{Statement}
\newtheorem{corollary}{Corollary}

\theoremstyle{definition}
\newtheorem{definition}{Definition}
\newtheorem{remark}{Remark}
\newtheorem*{remark*}{Remark}
\newtheorem*{example*}{Example}
\newtheorem{example}{Example}
\def\Z{{\Bbb Z}}
\def\R{{\Bbb R}}
\def\RP{{\Bbb R}\!{\rm P}}
\def\N{{\bf N}}
\def\L{{\bf L}}
\def\C{{\Bbb C}}
\def\CP{{\Bbb C}\!{\rm P}}
\def\A{{\bf A}}
\def\D{{\bf D}}
\def\Q{{\bf Q}}
\def\i{{\bf i}}
\def\j{{\bf j}}
\def\k{{\bf k}}
\def\E{{\bf H}}
\def\F{{\bf F}}
\def\J{{\bf J}}
\def\G{{\bf G}}
\def\I{{\bf I}}
\def\a{a}
\def\b{{\bf b}}
\def\bb{{\dot{b}}}
\def\e{{\bf e}}
\def\II{{\dot{\bf I}}}
\def\f{{\bf f}}
\def\d{{\bf d}}
\def\H{{\bf H}}
\def\id{{\operatorname{id}}}
\def\Ker{{\operatorname{Ker}}}
\def\const{{\operatorname{const}}}
\def\fr{{\operatorname{fr}}}
\def\st{{\operatorname{st}}}
\def\mod{{\operatorname{mod}\,}}
\def\cyl{{\operatorname{cyl}}}
\def\dist{{\operatorname{dist}}}
\def\sf{{\operatorname{sf}}}
\def\dim{{\operatorname{dim}}}
\def\dist{\operatorname{dist}}
\def\Img{\operatorname{Im}}

\maketitle

\begin{abstract}
In this paper, we introduce geometric technique
of working with skew-framed manifolds.
It allows us to study 
stable homotopy groups of some
Thom spaces by geometric means.
We schematically describe 
how our results 
(which are also of independent interest)
can be applied
to obtain a proof
of the Baum-Browder theorem
stating non-immersibility of $\RP^{10}$ to $\R^{15}$.
\end{abstract}

\section*{Introduction}

Our paper develops ideas of
the Lecture course
 \cite{Course} and we devote it to the memory
 of Yuri Petrovich Solovyov.

\smallskip

In \cite{A-F}, the authors 
suggested and partly realized a scheme
of a new geometric proof
of the Baum-Browder theorem
concerning non-immersibility of
projective space $\RP^{10}$ 
to euclidean space $\R^{15}$
\cite[Corollary (9.9)]{BB}. 
It turns out that 
the obstruction group 
for immersibility of $\RP^{10}$ to $\R^{15}$ 
is closely connected with the 
cobordism 
group
 $Imm^{sf(5)}(3,5)$ 
 of skew-framed immersions
of 3-dimensional manifolds to
euclidean space $\R^8$. 
More precisely: the group
$Imm^{sf(5)}(3,5)$ can be interpreted as a subgroup 
of the obstruction group
for the immersion mentioned above
 \cite{A-F}. 
In this paper we 
compute this cobordism group
and study some of its geometric-type properties.

Note that the Pontryagin-Thom construction \cite{pontr}, \cite{Thom} 
provides an isomorphism
 $Imm^{sf(5)} (3,5) \cong \Pi_8 (\RP^\infty _5)$, where
 $\RP_5^\infty  = \RP^{\infty} / \RP^{4}$ is
stunted projective space.
In these terms Theorem \ref{th1} 
states:
$$
\Pi _8 (\RP^8 _5) \cong \Z /2 \oplus \Z /2 ;
\quad
\Pi _8 (\RP^\infty _5) \cong \Z /2
.
$$
Results of
Theorems \ref{zamena}, \ref{th4} are
more technical. They are used in the proof
of the Baum-Browder theorem, which will be given in our subsequent paper.

\subsection*{Notation, conventions}

A number above an arrow
$\stackrel{(n)}{\to }$ means that
this map was defined in formula 
$(n)$.

The symbol $\cong $
has several meanings in the paper:
diffeomorphism of manifolds,
isomorphism of vector bundles,
isomorphism of groups.
Its concrete meaning is clear from the context.

All manifolds and maps are supposed to be smooth.
Manifolds can have boundaries.
Sometimes the dimension $m$ of a manifold $M$
is expressed in notation:
 $M^m$.

By $c$ we denote a constant map (to the point $c$).

Considering preimages of submanifolds,
we assume necessary transversality;
see e.g.
\cite{gp}.

A suspension
$Ef$
of an immersion
$f : M^n \looparrowright \R^{n+k}$
 is the immersion defined as the composition
$ M^n \stackrel{f}{\looparrowright }\R^{n+k}\subset \R^{n+k+1}$, where the second embedding is
the standard embedding onto the hyperplane $x_{n+k+1} =0 $.

The normal bundle of an immersion
 $f: M\looparrowright \R^{n+k} $ is denoted by $\nu (f)$, or simply by $\nu (M)$, 
if it is clear which immersion is meant.

For a submanifold
  $M\subset N$ we will use
  the evident isomorphism
$\nu (M) \cong \nu (M,N) \oplus (\nu (N) )|_M$.

We make free with $\varepsilon $ to denote
the trivial line bundle (over any space),
and with $\gamma $ to denote the tautological line bundle
over projective space (of arbitrary dimension).

For a vector bundle
 $\xi $ over a base space $B$
and a map $\kappa :M\to B$,
the pullback bundle
 $\kappa ^{*} (\xi )$
is defined in a standard way.
For $\xi = k\varepsilon $
we fix an evident isomorphism
$\kappa ^{*}(\xi ) \cong k\varepsilon $.

For
$N\in \mathbb N \cup \{ \infty \}$,
the $N$-dimensional real projective space
is denoted by $\RP^N$, 
and 
the stunted projective space is denoted by
 $\RP^{N}_k = \RP^{N}/ \RP^{k-1}$ .

Points of the total space
$ E( k_1\gamma \oplus k_2\varepsilon  ) $
of the bundle
$ k_1\gamma \oplus k_2\varepsilon   $
over $\RP^N $
are written as rows
\begin{equation}\label{points-of-total-space}
[ x; \lambda _1 ,\ldots , \lambda _{k_1 } , \mu _1 ,\ldots , \mu _{k_2} ],
\quad
\text{ where }
\quad
x\in S^N ,
 \lambda _1 ,\ldots , \lambda _{k_1 } , \mu _1 ,\ldots , \mu _{k_2} \in \R.
\end{equation}
Such a row is a pair of identified rows 
$$
( x; \lambda _1 ,\ldots , \lambda _{k_1 } , \mu _1 ,\ldots , \mu _{k_2} )
\sim
( -x; -\lambda _1 ,\ldots , -\lambda _{k_1 } , \mu _1 ,\ldots , \mu _{k_2} ) .
$$

By $\rho : S^N \to \RP^N$, $x\mapsto [x]= \{ x,-x\}$ we denote the standard
2-fold covering.

Using rows
$(x; \lambda _1 , \ldots , \lambda _k)$
to denote 
points of the space
$E({\rho }^{*}(k\gamma ) )$,
we define the isomorphism
${\rho }^{*}(k\gamma ) \cong k\varepsilon $
of bundles over $S^N$ 
using the formula
\begin{equation}\label{P-pullback}
E({\rho }^{*}(k\gamma ) )\to
E(k\varepsilon ),
\quad
(x; \lambda _1 , \ldots , \lambda _k) \mapsto
(x; \lambda _1 , \ldots , \lambda _k) .
\end{equation}

The embeddings  $\RP^n \subset \RP^{n+k}$
and the isomorphisms
$\nu (\RP^n , \RP^{n+k}) \cong k\gamma $
are standard ones.

A symbol 
$\Pi _n (X)$ denotes the $n$-th stable
homotopy group of space $X$; and
$\Pi _n = \Pi _n (S^0)$.
We will make use of some 
stable homotopy groups
of $\RP^\infty $ computed in
 \cite{L}.

Finally, let $I=[0,1]$.

\section{Basic Notions}

\subsection{Definition of the group $Imm ^{\xi ,B} (n,k)$}

Cobordism groups of framed embeddings were
defined and studied by L.S.~Pontryagin, see
\cite{pontr}.
Investigation
of cobordism groups of immersions with
additional structure on normal bundle
comes back to papers
\cite{Rokhlin}, \cite{Wells}, \cite{Watabe}, 
\cite{Vogel},
\cite{K-S};
see also \cite[Exercise 7.2.5]{Hirsch-book}.
(On some generalizations see e.g. \cite{Sz2}, \cite{Grant}.)

Let $\xi $ be a $k$-dimensional vector bundle over a base space
$B$, and $k\geqslant 1$.
Let $n\geqslant 0$.

\begin{definition}
A normal
$\xi $-structure
for
an immersion of an $n$-dimensional
compact manifold 
$f:M^n\looparrowright \R^{n+k}$ 
is a continuous map
$\kappa : M\to B$ together with
an isomorphism
$\Xi : \nu (f) \cong \kappa ^{*} \xi $.

Such an object is called
an immersion with normal
 $(\xi ,B)$-structure,
or shortly a $(\xi ,B)$-immersion.
It is denoted by 
triples $(f_M,\kappa _M, \Xi _M)$.
Sometimes is clear which manifold is meant, 
in these cases we write simply $(f,\kappa , \Xi )$.
\end{definition}

\begin{remark}\label{xi-structure-as-cd}
Equivalently, a $(\xi ,B)$-structure
$\Xi : \nu (f) \cong \kappa ^{*} (\xi )$ over $M$
can be thought of as a fiberwise isomorphism of bundles
$$
\begin{CD}
\nu (f)  @>>> \xi  \\
@VVV @VVV \\
M @> \kappa >> B
\end{CD}
$$
\end{remark}

\begin{remark}\label{diff-and-imm}
For a covering $p:\tilde M\to M$ 
(in particular, for a diffeomorphism
 $\varphi :M\to M$)
and an immersion $f: M\looparrowright \R^{n+k}$,
there is a natural fiberwise isomorphism
of normal bundles
$$
\begin{CD}
\nu (f\circ p) @>>> \nu (f) \\
@VVV @VVV \\
\tilde M @> p >> M
\end{CD}
$$
The corresponding isomorphism
$\nu (f\circ p) \cong p^{*} (\nu (f))$
of bundles over $\tilde M$
will be denoted by $p'$.

Hence an isomorphism
$\Xi : \nu (f) \cong \kappa ^{*} (\xi )$
induces an
isomorphism
$p^{*}(\Xi ) \circ p' : \nu (f\circ p ) \cong 
p^{*} (\nu (f) ) \cong
p^{*} (\kappa ^{*} (\xi ))$.
\end{remark}

\begin{definition}\label{bordant-triples}
Two immersions with normal
 $(\xi ,B)$-structures
$(f_0 : M^n_0 \looparrowright \R^{n+k},\kappa _0,\Xi _0)$, $(f_1 : M^n_1 \looparrowright \R^{n+k},\kappa _1,\Xi _1)$ are called cobordant
if there exist
\begin{enumerate}
\item 
a compact $(n+1)$-dimensional manifold  $\mathcal M$,
\item
a diffeomorphism
$\varphi _0\sqcup \varphi _1 : M_0 \sqcup M_1 \to \partial \mathcal M$,
\item
an immersion
$\mathcal F : \mathcal M \looparrowright \R^{n+k} \times I$ orthogonal to the boundary 
 $\R^{n+k}\times \partial I$ of the strip
$\R^{n+k} \times  I$ at any point of  $\partial M$; also, 
 the restriction
${\mathcal F}|_{\varphi _s (M_s)} $
must
equal
$M_s \stackrel{f_s}{\looparrowright } \R^{n+k}
\subset \R^{n+k} \times \{ s \} $
for $s=0,1$;
\item
a continuous map
$\kappa : \mathcal M \to B $
such that 
the composition map $\kappa \circ \varphi _s $ equals $ \kappa _s : 
M_s \to B$ for $s=0,1$; 
\item
an isomorphism
$\Xi : \nu (\mathcal F) \cong \kappa ^{*} (\xi )$,
such that 
the bundle
$
\nu (f_s) =
\varphi _s^{*}(\nu (\mathcal F) )
\stackrel{\varphi _s^{*} \circ \Xi }{\cong }
\varphi _s^{*} (\kappa ^{*} \xi )
= \kappa _s ^{*} \xi
$
over $M_s$
coincides with $\Xi _s$ for $s=0,1$.
\end{enumerate}

In other words, the triple
$(\mathcal F , \kappa , \Xi ) $ 
provides a cobordism of
$(f_0,\kappa _0,\Xi _0)$ and $(f_1,\kappa _1,\Xi _1) $;
we write
$$
\partial (\mathcal F , \kappa , \Xi )  = 
(f_0,\kappa _0,\Xi _0) \sqcup (f_1,\kappa _1,\Xi _1) .
$$
\end{definition}

In what follows, we omit
the diffeomorphism symbol
$\varphi_0\sqcup\varphi_1$,
identifying $\partial \mathcal M$ 
with $ M_0\sqcup M_1$.

By
standard arguments, this cobordism relation
is reflexive, symmetric, and transitive.

Denote by
$Imm^{\xi ,B} (m,k)$
 the set of cobordism classes
of immersions of
{\bf closed} $m$-dimensional
manifolds with normal
 $\xi $-structure.
Sometimes we 
inaccurately speak of 
$(\xi , B)$-immersions as elements 
of this group.

The sum of two cobordism classes is defined as
the union of disjoint 
representatives; thus
$Imm^{\xi ,B} (m,k)$ becomes an abelian group.

The Whitney Theorem \cite[Theorem 1.3.5]{Hirsch-book}
and the Pontryagin-Thom construction \cite{pontr},
\cite{Thom}
(see also \cite[Exercise 7.2.5]{Hirsch-book}) 
imply

\begin{proposition}\label{emdedding-case}
Suppose $k\geqslant n+2$. Then
$$
Imm^{\xi , B} (n,k) 
\cong Emb^{\xi , B} (n,k)
\cong \pi _{n+k} (T(\xi )) .
$$
\end{proposition}

Together with Hirsch theorem
 \cite[Theorem 6.4]{Hirsch-paper} this
gives (see \cite[Theorem~1]{Wells}, 
\cite[Theorem 1.1]{K-S}):

\begin{proposition}\label{Thom-space}
Let $B$ be a finite complex.
Then
there is an isomorphism
$Imm^{\xi ,B} (n,k)\cong \Pi _{n+k} (T(\xi ))$.
\end{proposition}

\subsection{Important particular cases}

We have given general definition of immersion
with normal $(\xi ,B)$-structure.
Let us enumerate
several particular cases which will
be investigated in the present work
(see \cite{A-E}, \cite{A-F}).

\begin{enumerate}
\item
For $B=pt $, $\xi = k\varepsilon $ the group
$Imm^{k\varepsilon , pt }  (n,k) $
is denoted by
$Imm^{fr(k)} (n,k) $;
we call it the group of
framed immersions.
Elements of this group are
represented by pairs of the form
 $(f : M^n\looparrowright \R^{n+k}, \Xi : \nu (f) \cong k\varepsilon )$.
The Pontryagin-Thom construction
together with Hirsch results
provides for $k\geqslant 1$
an isomorphism
$$
Imm^{fr(k)} (n,k) \cong \Pi_{n} .
$$
\item
For $B=\RP^N  $, $\xi = k\gamma $ the group
$Imm^{k\gamma , \RP^N }  (n,k) $ 
will be denoted by $Imm^{sf(k) , \RP^N} (n,k)$,
we call it the group of {\it skew-framed} 
$\RP^N$-controlled
immersions. 
For $N=\infty $ 
this group is denoted by $Imm^{sf(k)} (n,k)$,
and we speak simply of the 
skew-framed immersions
group.
The chain of standard inclusions
$\RP^0\subset \RP^1 \subset \ldots \subset \RP^\infty $
induces a chain of homomorphisms
(by the cell approximation theorem, 
the first map in the second row is surjective, and
all other maps are bijective)
$$
Imm^{sf(k),\RP^0} (n,k) 
= Imm^{fr(k)} (n,k)
\to 
Imm^{sf(k),\RP^1} (n,k) 
\to 
\ldots
\to
$$
\begin{equation}\label{control}
Imm^{sf(k),\RP^{n}} (n,k) 
\twoheadrightarrow
Imm^{sf(k),\RP^{n+1}} (n,k) 
\cong
\ldots
\cong
Imm^{sf(k)} (n,k) .
\end{equation}

Statement
 \ref{Thom-space}
and
\cite[Theorem 15.1.8]{Husemoller} 
provide the isomorphisms
$$
Imm^{sf(k) , \RP^N } (n,k) \cong \Pi_{n+k} (\RP^{k+N} _k) ,
\quad
Imm^{sf(k)} (n,k) \cong \Pi_{n+k} (\RP^\infty _k) .
$$
\item
For $B=\RP^N $, 
$\xi = k_1\gamma \oplus k_2\varepsilon $,
$ k_2 \geqslant 1$ 
we have a suspension isomorphism
\begin{equation}\label{suspension}
E: Imm^{sf(k_1)\times fr(k_2-1) , \RP^N} (n,k_1+k_2-1) 
\to 
Imm^{sf(k_1)\times fr(k_2) , \RP^N } (n,k_1+k_2) ,
\end{equation}
for $k_1+k_2 \geqslant 2$ it is an isomorphism 
by \cite[Theorem 6.4]{Hirsch-paper}.
The group
$Imm^{sf(k_1)\times fr(k_2) , \RP^\infty } (n,k_1+k_2)$
is denoted shortly by
$Imm^{sf(k_1)\times fr(k_2)  } (n,k_1+k_2)$.
\end{enumerate}

Injection of
the structure groups
 $\{0  \} \subset \Z /2 $
induces a homomorphism
\begin{equation}\label{t-skashivanie}
Imm^{fr(k)} (n,k) \to Imm^{sf(k)} (n,k) 
\end{equation}

\begin{example}\label{sf-zero}
The map $\iota : Imm^{fr(k)} (n,k) \to Imm^{fr(k)\times sf(0)} (n,k) $ such that $(f,\Xi ) \mapsto (f,\const , \Xi )$
is a monomorphism.

The map $r: Imm^{fr(k)\times sf(0)} (n,k) \to Imm^{fr(k)} (n,k) $ such that $(f,\kappa , \Xi : \nu (f) \cong \kappa ^{*} (k\varepsilon \oplus 0\gamma )) 
\mapsto (f, \Psi : \nu (f) \stackrel{\Xi }{\cong }\kappa ^{*} (k\varepsilon ) \cong k\varepsilon )$
is an epimorphism.

The composition $r \circ \iota : Imm^{fr(k)} (n,k) \to Imm^{fr(k)} (n,k)$ is the identity map.
Therefore
the group $Imm^{fr(k)} (n,k) $ 
is a direct summand of the group
$Imm^{fr(k)\times sf(0)} (n,k) $.
\end{example}

\subsection{Simple modifications of the triples $(f,\kappa , \Xi )$}

In this section we describe several simple
procedures which will be applied in the proofs below.
The reader can verify their properties himself or herself.

\begin{proposition}\label{simple-changes}
For arbitrary $B$, $\xi $ we have:
\\
1) 
for any 
diffeomorphism $\varphi : M\to M$ the
$(\xi ,B)$-immersions
$(f: M^n \looparrowright \R^{n+k},\kappa , \Xi )$ and
$(f\circ \varphi ,
\kappa \circ \varphi ,
\varphi ^{*}( \Xi ) \circ \varphi '   )$
are cobordant 
[concerning $\varphi '$ see Remark
\ref{diff-and-imm}].
\\
2) Suppose
$(f,\kappa , \Xi _0)$ is a
$(\xi ,B)$-immersion, and
 $\{ \Xi _t \}$, $t\in I$ is a family of bundle isomorphisms $\nu (f) \cong \kappa ^{*} (\xi )$.
Then the $(\xi ,B)$-immersions
$(f,\kappa , \Xi _0)$ and
$(f ,\kappa , \Xi _1)$ 
are cobordant.
\\
3)
Let $H: M\times I \to \R^N\times I$ be a regular homotopy
of immersions $f_0\simeq f_1$ such that
$f_t=f_0$ for $t\in [0,\frac{1}{10}] $,
and  $f_t=f_1$ for $t\in [\frac{9}{10},1] $.
Suppose $\kappa : M\to B $ 
is  a continuous map 
and 
$\Xi _0 : \nu (H)|_{M\times \{ 0 \}} = \nu (f_0) \cong \kappa ^{*} ( \xi )$ is an isomorphism.
Let
$\Xi : \nu (H) \cong (\kappa \circ p)^{*}( \xi )$ be any isomorphism extending the isomorphism $\Xi _0$;
here $p:M\times I \to M$ is the projection map.
Let $\Xi _1 = \Xi |_{M\times \{1 \} }  : \nu(H)|_{M\times \{1\}} = \nu (f_1)
\cong \kappa ^{*} ( \xi )$ be ``the resulting'' 
isomorphism.
Then the $(\xi , B)$-immersions
$(f_0,\kappa , \Xi _0)$ and
$(f_1,\kappa , \Xi _1)$ 
are cobordant.
\end{proposition}

Now let us consider the case
$B=\RP^N $, $\xi = k_1\gamma \oplus k_2\varepsilon $.
In this situation, we can 
 ``spread out'' the $\xi $-structure
following a homotopy of the map $\kappa $.
Such spreading generalizes
the procedure described by Pontryagin for
framed embeddings.

So, let a $(k_1\gamma \oplus k_2\varepsilon  ,\RP^N )$-immersion
$(f,\kappa _0,\Xi _0)$ and a homotopy $\{ \kappa _t \}$ be given.

We describe a method to construct 
the corresponding family $\{ \Xi _t : \nu (f) \cong \kappa _t^{*} (k_1\gamma \oplus k_2\varepsilon  ) \} $.

In accordance with Remark
\ref{xi-structure-as-cd}
replace the isomorphism $\Xi _0:\nu(f) \cong \kappa _0^{*} (\xi )$
by the diagram
$$
\begin{CD}
E(\nu (f))  @> \Psi _0 >>  E(\xi = k_1\gamma \oplus k_2\varepsilon  ) \\
@V \pi VV @VVV \\
M @> \kappa _0 >> \RP^N
\end{CD}
$$
where $\Psi _0$ is a fiberwise isomorphism.

Let us define the isomorphism $\Psi _t$ in the diagram
$$
\begin{CD}
E(\nu (f))  @> \Psi _t >>  E( k_1\gamma \oplus k_2\varepsilon  ) \\
@V\pi VV @VVV \\
M @> \kappa _t >> \RP^N
\end{CD}
$$
Take an arbitrary point $X\in E(\nu(f))$.
If
$$
\Psi _0 (X) = 
[ \kappa _0 (\pi (X)); \lambda _1 (X) ,\ldots , \lambda _{k_1 } (X) , \mu _1 (X) ,\ldots , \mu _{k_2} (X) ]
$$
then put
$$
\Psi _t (X) =
[ \kappa _t (\pi ( X)) ; \lambda _1 (X),\ldots , \lambda _{k_1 } (X), \mu _1 (X),\ldots , \mu _{k_2} (X)].
$$
Denote  by $\Xi _t$
the corresponding isomorphism 
$\nu (f) \cong \kappa _t ^{*} (k_1\gamma \oplus k_2 \varepsilon )$.

\begin{proposition}\label{kappa-homotopy}
Suppose $B=\RP^N $, $\xi = k_1\gamma \oplus k_2\varepsilon $. Then
the above construction provides a
 $(k_1\gamma \oplus k_2\varepsilon  ,\RP^N )$-immersion 
 $(f,\kappa _1,\Xi _1)$
 which is
 cobordant to the original
immersion $(f,\kappa _0,\Xi _0)$.
\end{proposition}

\begin{example}\label{simple-sphere}
For the standard embedding $\mathbb E :S^3 \subset \R^4$,
the unit exterior normal vector
gives a trivialization
$\Phi : \nu (\mathbb E) \cong c^{*} (\gamma )$,
where $c:S^3 \to c\in \RP^\infty $ is the constant map.

The pair $(\mathbb E , \Phi )$ 
represents zero element of the group
$Imm^{fr(1)} (3,1)$;
the triple $(\mathbb E , c ,\Phi )$
represents the zero element both in
$Imm^{sf(1),\RP^3} (3,1)$
and in $Imm^{sf(1)}(3,1)$.
\end{example}

\begin{example}\label{replace-c-by-covering}
To continue Example \ref{simple-sphere},
take a homotopy $H : S^3 \times I \to \RP^\infty $
between the maps $c$ and
$S^3 \stackrel{\rho }{\to } \RP^3 \subset \RP^\infty $;
here $\rho $ is the standard 2-covering.
Statement
\ref{kappa-homotopy}
uses 
the isomorphism
$\Phi : \nu (\mathbb E) \cong c^{*}(\gamma )$
to construct the isomorphism
$\Xi _{\mathbb E }: \nu (\mathbb E ) 
\cong \mathcal \rho ^{*} (\gamma )$.

In the group $Imm^{sf(1)} (3,1)$ the triple
$(\mathbb E ,\rho  , \Xi _{\mathbb E})$ 
is cobordant to the original triple $(\mathbb E , c , \Phi )$,
that is, it represents the zero element.

But since the image of the homotopy $H$ is not contained in the subspace $\RP^3\subset \RP^\infty $,
we can not assert that the triple
$(\mathbb E ,\rho  , \Xi _{\mathbb E})$ 
is cobordant to the original triple
$(\mathbb E , c , \Phi )$ in the group
$Imm^{sf(1), \RP^3} (3,1)$.
It turns out that the triple
$(\mathbb E ,\rho  , \Xi _{\mathbb E})$ represents  a non-zero element in the group
$Imm^{sf(1), \RP^3} (3,1)$
(see Remark \ref{E-5}).
\end{example}

\begin{remark}\label{perevorot-1}
Formally, the triple $(\mathbb E ,\rho  , \Xi _{\mathbb E})$
defined in Example
\ref{replace-c-by-covering}
depends on the choice of the homotopy
$H$.
But it turns out that its cobordism class
does depend on the choice of homotopy.
In fact, 
since $S^3$ is simply connected,
for any choice of homotopy 
$H$ 
the resulting triple will be cobordant
either to the triple $(\mathbb E ,\rho  , \Xi _{\mathbb E})$ or to the triple
$(\mathbb E ,\rho  , - \Xi _{\mathbb E})$, where
 $-\Xi _{\mathbb E}$ is the trivialization obtained
 from the trivialization
$\Xi _{\mathbb E}$ by multiplication
by $-1$ in each fiber.
The triples
$(\mathbb E ,\rho  , \Xi _{\mathbb E})$ and
$(\mathbb E ,\rho  , - \Xi _{\mathbb E})$
are cobordant in the group
$Imm^{sf(1), \RP^3} (3,1)$ by
Statement
 \ref{reflection-all-sf}.
\end{remark}

\section{Some geometric constructions}

\subsection{Construction of inverse element}

Let $R: \R^d \to \R^d $ be the reflection in the
hyperplane
$x_{d} =0$.

For an arbitrary immersion $F:M\looparrowright \R^d$
define
an isomorphism $R_{\nu(F) }$ 
in a natural way:
$$
\begin{CD}
\nu(F) @> R_{\nu (F)  } >> \nu (R\circ F) \\
@VVV @VVV \\
M @> \id >> M
\end{CD}
$$

If we are also given an isomorphism $\Xi : \nu (F) \cong  \kappa ^{*}( 
k_1 \gamma \oplus k_2  \varepsilon ) $,
then we can define
the composite isomorphism
$
\Xi \circ (R_{\nu(F)})^{-1}:
\nu(R\circ F)\stackrel{(R_{\nu (F) })^{-1}}{\to  }
\nu (F) \stackrel{\Xi }{\to }
\kappa ^{*}(k_1\gamma \oplus k_2\varepsilon ) $.

Then, the following simple assertion holds.

\begin{proposition}\label{reflection}
For $k_2\geqslant 1$ and $N\in \mathbb N\cup \{ \infty \}$
the triples
$(F,\kappa ,\Xi )$ and
$( R\circ F,
\kappa , 
\Xi \circ (R_{\nu (F) })^{-1} )  $
represent
mutually inverse elements of the group 
   $ Imm ^{sf (k_1) \times fr (k_2), \RP^N} (n, k_1+k_2)$.
\end{proposition}

Let $k_2 \geqslant 1$;
define by $\mathfrak r $ 
the isomorphism of the bundle
$k_1 \gamma \oplus k_2 \varepsilon $ over $\RP^N$
which 
overturns the last summand $\varepsilon $:
$$
\xymatrix{
[x; \lambda _1, \ldots , \lambda _{k_1}, \mu _1 , \ldots , \mu _{k_2}]  \ar[rr] \ar[dr] & &
[x; \lambda _1, \ldots , \lambda _{k_1}, \mu _1 , \ldots , -\mu _{k_2}] 
  \ar[dl]  \\
& [x] & 
}
$$

\begin{proposition}\label{reflection-2}
For $k_2\geqslant 1$ and $N\in \mathbb N\cup \{ \infty \}$
the triples
$(F,\kappa ,\Xi )$ and
$(F,
\kappa , 
 \Psi  )  $,
where
 $\Psi $ is the composition
$$
 \nu (F) \stackrel{\Xi }{\cong }\kappa ^{*}
 ( k_1 \gamma \oplus k_2 \varepsilon ) 
\stackrel{\kappa^{*} (\mathfrak r )}{\cong }
\kappa ^{*}( k_1 \gamma \oplus k_2 \varepsilon ) ,
$$ 
represent mutually inverse elements
of the group
$ Imm ^{sf (k_1) \times fr (k_2),\RP^N } (n, k_1+k_2)$.
\end{proposition}

\begin{proof}
By \cite[Theorem 6.4]{Hirsch-paper},
we can apply regular homotopy and assume that
the image $F(M)$ lies in the hyperplane
$x_d=0$ of space $\R^{n + k_1+k_2}$.
(We assume that the map $\kappa $ is left unchanged,
and the isomorphism
$\Xi $ has changed as in item 3) of Statement \ref{simple-changes}; it is denoted by the same symbol $\Xi $. 
The cobordism class has not changed.)
Let $R:\R^{n + k_1+k_2}\to \R^{n + k_1+k_2}$ be the reflection in the hyperplane  $x_d=0$.
By Statement \ref{reflection}
the 
element represented by the
triple
$( R\circ F,
\kappa , 
 \Xi \circ (R_{\nu (F) })^{-1} )  $
is inverse to the original one. 
But  $R\circ F = F$ and $\Xi \circ (R_{\nu (F) })^{-1} 
= \Psi $.
\end{proof}

\subsection{
Antipodal transformation of a skew-framing keeps an element in the cobordism group unchanged
}

For arbitrary vector bundle $\xi $ 
over base space $B$
define the fiberwise antipodal involution.
It is the automorphism ${\Upsilon } _{\xi }$
such that 
${\Upsilon } _{\xi } (v) = -v$
for any vector $v\in F_{x}$ 
belonging to the fiber over $x\in B$.
[If the bundle under consideration is clear from the context,
then we will omit the symbol $\xi $.]
It is easily seen that for the map $\kappa :M\to B$
the equality holds
$
\kappa ^{*} \circ {\Upsilon } _{\xi } = {\Upsilon } _{\kappa ^{*} (\xi )} .
$

Statement \ref{reflection-2} implies
(compare \cite[p.206]{B-S})

\begin{proposition}\label{reflection-all-fr}
In the group $Imm^{fr(k)} (n,k)$
the pairs $(F, \Xi )$
and $(F,\Theta )$,
where 
$\Theta : \nu (F) \stackrel{\Xi }{\cong } k\varepsilon 
\stackrel{ \Upsilon   }{\cong } k\varepsilon  $,
represent:

for odd $k$, mutually inverse elements,

for even $k$, equal elements.
\end{proposition}

In this section we prove

\begin{proposition}\label{reflection-all-sf}
In the group
$Imm^{sf(k) } ( n , k)$ 
the triples
$(F, \kappa , \Xi )$ and
$(F, \kappa , \Theta )$, where
$\Theta : \nu (F) \stackrel{\Xi }{\cong } \kappa ^{*} (k\gamma ) 
\stackrel{\Upsilon }{\cong } \kappa ^{*} (k\gamma ) $,
represent equal elements.

For odd $n$, the same holds true
if these triples are considered as representatives
of the group $Imm^{sf(k),\RP^n} (n,k)$.
\end{proposition}

For the proof we need some preliminary remarks.

Projective space of odd dimension $\RP^{2d+1}$
consists of elements of the form
$ [x ] = \{  x , - x \}$,
where $ x \in S^{2d+1}$.
Since $S^{2d+1} \subset \R^{2d+2} \cong \C^{d+1}$,
we can write
$x $ as an array of complex numbers:
$ (z_1, \ldots , z_{d+1})$.

Define an isotopy
$h_t : \RP^{2d+1} \to \RP^{2d+1}$, $t\in I$ by
the formula
$$
h_t ( [ (z_1, \ldots , z_{d+1}) ] ) 
=  [( e^{\pi i t} \cdot z_1 , \ldots , e^{\pi i t} \cdot z_{d+1} )] .
$$
It is clear that $h_0=h_1=\id $.

Points of the total space
 $E(\gamma )$ 
 are written as arrays of the form
$[z_1 , \ldots , z_{d+1} ; \lambda ]$, see
(\ref{points-of-total-space}).

For $t\in I$
define a map
$H_t : E(\gamma ) \to E(\gamma )$ 
by the formula
\begin{equation}\label{isotopy}
H_t ( [z_1 , \ldots , z_{d+1} ; \lambda ] ) =
[  e^{\pi i t} \cdot z_1 , \ldots , e^{\pi i t} \cdot z_{d+1}  ; \lambda  ] .
\end{equation}

Next statement has straightforward proof:

\begin{lemma}
1. The map $H_t$ is well-defined.
\\
2. The diagram
$$
\begin{CD}
E(\gamma )  @> H_t >> E(\gamma )  \\
@VVV @VVV \\
\RP^{2d+1} @> h_t >> \RP^{2d+1}
\end{CD}
$$
commutes for $t\in I$.
\\
3. 
For induced commutative diagram
$$
\xymatrix{
E(\gamma ) \ar[rr]^{L_t} \ar[dr]^{} & &
E(h_t^{* } \gamma ) \ar[dl]_{}  \\
& \RP^{2d+1} & 
}
$$
we have $L_0 = \id $, $L_1 = {\Upsilon }_{\gamma } $.
\end{lemma}

Now let us prove Statement \ref{reflection-all-sf}.

Let us construct a cobordism 
$(\hat F , \hat \kappa , \hat \Xi )$ 
of the triples under consideration.
Take a manifold $\hat M = M\times I$
and 
the product immersion $\hat F = F\times \id _I:M\times I \looparrowright \R^{n+k} \times I$.

Let
$$
N = 
\begin{cases}
n,&\text{for odd $n$;}\\
n+1,&\text{for even $n$.}
\end{cases}
$$

By the cell approximation theorem 
we can assume that
$\kappa (M) \subset \RP^n$.
(In the case 
of the 
$\RP^n$-controlled
group, 
this
inclusion holds automatically.)

Define the map
$\hat \kappa : M\times I \to \RP^N  $ 
by the formula
$$
\hat \kappa |_{M\times \{ t \} } = 
h_t \circ \kappa \circ p ,
$$
where 
$p:M\times I \to M$ is the projection, and 
$h_t : \RP^N\to \RP^N $ is defined by the formula (\ref{isotopy}).

Note that $\hat \kappa |_{M\times \{ 0 \} } = \kappa $,
$\hat \kappa |_{M\times \{ 1 \} } = \kappa $, 
as required.

Define the isomorphism
$\hat \Xi : \nu (\hat F ) \cong \hat \kappa ^{*} (k\gamma ) $.

To do this, for each $t\in I$ define
an isomorphism
$$
\hat \Xi _t : \nu (\hat F|_{M\times \{t \} }) 
\cong (\hat \kappa |_{M\times \{t \} }) ^{*} (k\gamma ) .
$$

Let
$\hat \Xi _t $ be the composition
\begin{equation}
\begin{split}
\nu (\hat F|_{M\times \{t \} })  & \cong 
\nu (\hat F|_{M\times \{0 \} }) \stackrel{\Xi }{\cong }
\kappa ^{*}(k\gamma )
\stackrel{\kappa ^{*} (L_t ) }{\cong }
\\
&
\cong
\kappa ^{*} (h_t^{*} (k\gamma ))
= (h_t \circ \kappa )^{*} (k\gamma )
= 
(\hat \kappa |_{M\times \{t \} }) ^{*} (k\gamma ) .
\end{split}
\end{equation}

Note that
$\hat \Xi |_{M\times \{ 0 \} } = \Xi $
and
$\hat \Xi |_{M\times \{ 1 \} } = \Theta $, as required.
\hfill $\qed $

Statements \ref{reflection-all-fr},
\ref{reflection-all-sf}
imply

\begin{corollary}\label{order-2}
Let $k$ be an odd integer.
Then the homomorphism
(\ref{t-skashivanie})
$$
\phi : Imm^{fr(k) } (n,k) \to Imm^{sf(k) } (n,k)
$$
satisfies
$2\Img \phi = \{ 0 \}$.
\end{corollary}

\subsection{Transfer homomorphism}

Recall the definition of the
transfer homomorphism \cite{A-E}
\begin{equation}\label{transfer1}
Imm^{sf(k), \RP^N } (n,k)  \to Imm^{fr(k)} (n,k)  .
\end{equation}
For the bundle $\gamma $ over $\RP^N$ 
let $S(\gamma )$ be the corresponding
spherical bundle over $\RP^N$,
and $\pi : S(\gamma ) \to \RP^N$ the projection map.
The pullback $\pi ^{*}(\gamma )$ is
the trivial line bundle over $S(\gamma )$.
Fix an isomorphism
$\tau : \pi ^{*}(\gamma ) \cong \varepsilon $.

Suppose we are given a triple
$(f:M\looparrowright \R^{n+k} ,\kappa : M\to \RP^N, 
\Xi : \nu (f) \cong \kappa ^{*} (k\gamma ) )$.

Take $\tilde M = S( \kappa ^{*} (\gamma )) $. 
It is an $n$-manifold,
and $p=\kappa ^{*}(\pi ) : \tilde M \to M$ 
is its 2-fold covering:
$$
\begin{CD}
\tilde M = S(\kappa ^{*}(\gamma )) @> \tilde \kappa >> S(\gamma )  @<<< \pi ^{*}(\gamma ) \stackrel{\tau }{\cong } \varepsilon \\
@V p = \kappa ^{*}(\pi )VV @V \pi VV @VVV \\
M @> \kappa >> \RP^N @<<< \gamma 
\end{CD}
$$
There is a sequence of isomorphisms
\begin{equation}\label{transfer-3}
p^{*} \kappa ^{*} (\gamma ) \cong
( \kappa \circ p) ^{*} (\gamma ) \cong
(\pi \circ \tilde \kappa )^{*}(\gamma )
= \tilde \kappa ^{*} (\pi ^{*}(\gamma ))
\stackrel{\kappa ^{*}(\tau )}{\cong }
\tilde \kappa ^{*}(\varepsilon )
\cong \varepsilon .
\end{equation}

In Remark
\ref{diff-and-imm},
for a given isomorphism
$\Xi : \nu (f) \cong \kappa ^{*} (k\gamma )$,
we give the method 
to construct the isomorphism
$p^{*}(\Xi ) \circ p' :
\nu (f\circ p) \cong p^{*} (\kappa ^{*} (k\gamma ))$.
Taking its composition with
$k$-th iterate of the isomorphism
(\ref{transfer-3}),
we obtain the isomorphism
$\tilde \Xi : \nu (\tilde f) \cong k\varepsilon  $.

Now, to a given triple
$(f:M\looparrowright \R^{n+k} ,\kappa : M\to \RP^N, 
\Xi : \nu (f) \cong \kappa ^{*} (k\gamma ) )$,
assign
the pair
$(\tilde f = f\circ p:\tilde M\looparrowright \R^{n+k} , 
\tilde \Xi : \nu (\tilde f) \cong k\varepsilon  )$.

This formula defines the homomorphism (\ref{transfer1}).

\begin{example}
The transfer homomorphism $\Z /2\cong Imm^{sf(1)} (0,1)  \to Imm^{fr(1)} (0,1)\cong \Z $ is zero.
\end{example}

\begin{example}
Composition of the homomorphism (\ref{t-skashivanie}) and tranfser
$$
Imm^{fr(k)} (n,k)  \to Imm^{sf(k)} (n,k)  \to Imm^{fr(k)} (n,k)  
$$
is for even $k$ a multiplication by $2$;
\\
and for odd $k$, a zero map.
\end{example}

\begin{example}
The transfer homomorphism
$$
Imm^{sf(5)} (3,5)  \to Imm^{fr(5)} (3,5)  
$$
is zero.

In fact, by the previous example
the composition
$$
Imm^{fr(5)} (3,5)  \to Imm^{sf(5)} (3,5)  \to Imm^{fr(5)} (3,5)  
$$
is zero. 
But the first homomorphism
$\Z/24 \cong Imm^{fr(5)} (3,5) \to Imm^{sf(5)} (3,5) \cong \Z/2$ is surjective
(see Theorem \ref{th1} and \cite[Theorem~1]{A-F}). Hence, the considered transfer is zero.
\end{example}

\section{Low-dimensional examples}

\subsection{Dimension $0$}

\begin{example}
\begin{equation}\label{fr-0-1}
Imm ^{fr(1)} (0,1) \cong \Pi _0 \cong \Z ,
\end{equation}
however
\begin{equation}\label{sf-0-1}
Imm ^{sf(1)} (0,1) \cong \Pi_1(\RP^{\infty} ) \cong \Z /2 ,
\end{equation}
since $\RP^{\infty}$ is the Thom space
of the universal line
$O(1)$-bundle.

Although
the elements of both group
are finite sets of points, 
and cobordisms are sets of segments 
(over which each line bundle is trivial),
in the case 
of the group of skew-framed immersions
we are to construct a characteristic map
$\kappa : M\to \RP^{\infty } $ on the cobordism.
Let us underline that
here additional possibilities arise,
in distiction from the case of framed immersions
(where $\kappa $ is a constant map).

Let us explain shortly the formulaes
(\ref{fr-0-1}) and
(\ref{sf-0-1}).

The generator $a_0$ of the first group
can be represented 
by an immersion of the point,
i.e. by the pair
$(f_0: p_0 \mapsto 0\in \R^1 ,
\Xi _0 : \nu (f_0) \cong \varepsilon )$, where $\Xi _0$ is an isomorphism (arbitrary, but fixed).

The homomorphism defined in
(\ref{t-skashivanie}) 
$$
\phi : Imm ^{fr(1)} (0,1) \to Imm ^{sf(1)} (0,1) 
$$
is surjective, and $\phi (a_0 ) \neq 0$.
Statement \ref{reflection-all-sf}
implies that $\phi (a) = \phi (-a) $; hence $Imm ^{sf(1)} (0,1) \cong \Z / 2$.
\end{example}

\subsection{Dimension $1$}

\begin{example}
The homomorphism (\ref{t-skashivanie})
$$
\Z/2\cong \Pi _1 \cong Imm^{fr(1)} (1,1) \stackrel{\phi }{\to } Imm^{sf(1)} (1,1) \cong \Pi _{2} (\RP^\infty ) 
\cong \Z /2
$$
is an isomorphism.

Argument 1: Each compact 1-manifold
is a union of circles.
Let an element $a_1\in Imm^{sf(1)} (1,1)$ be represented by the triple
$(F:S^1 \looparrowright \R^2, \kappa : S^1 \to \RP^\infty ,
\Xi : \nu (F) \cong \kappa ^{*}(\gamma ))$.
Since the bundle $\nu (F)$ is trivial, 
the map $\kappa $ is homotopic to a constant map.
Thus $a_1\in \Img \phi $, hence $\phi $ is an isomorphism.

Argument 2:
Consider the commutative diagram
$$
\xymatrix{
\Z/2 = Imm^{fr(1)} (1,1) \ar[rr]^{\phi } \ar[dr]^{q^{fr}} & &
 Imm^{sf(1)} (1,1) = \Z / 2 \ar[dl]_{q^{sf}}  \\
& \Z /2 & 
}
$$
The homomorphisms 
$q^{fr}$, $q^{sf}$ 
assign to 
a given general position immersion
a number of its double points mod 2 \cite[Proposition 2.2]{Eccles}.
Investigation of the ``figure eight'' immersion
shows that both $q^{fr}$ and $q^{sf}$ are epimorphisms, hence, are isomorphims.
Hence $\phi $ is also an isomorphism.
\end{example}

Let us introduce a notation
for the two next examples.
For a line bundle $\zeta $ over a 
closed 1-dimensional manifold $M = S^1 _1 \sqcup \ldots
\sqcup S^1_k$
we define a modulo 2 residue
$$
s (\zeta ) =
w_1 (\zeta |_{S^1_1}) + \ldots + w_1 (\zeta |_{S^1_k})
\in \Z /2.
$$
Recall that elements of the group
$
Imm^{fr(1)\times sf(0)} (1,1)
$
are represented by triples of the form
$(f,\kappa , \Xi )$, where
$f:M \looparrowright \R^2$ 
is an immersion of a
closed 1-manifold,
$\kappa  : M \to \RP^\infty $ is a continuous map,
and
$\Xi : \nu (f) \cong \kappa ^{*} (1\varepsilon \oplus 0\gamma )$ is an isomorphism.

\begin{example}
There is an isomorphism
$$
Imm^{fr(1)\times sf(0)} (1,1) \to Imm^{fr(1)} (1,1) \oplus \Z /2 ,
$$
it can be defined by the formula
$$
(f,\kappa , \Xi : \nu (f) \cong \kappa ^{*} (1\varepsilon \oplus 0\gamma ) ) 
\mapsto ((f,\Psi : \nu (f) \stackrel{\Xi }{\cong }\kappa ^{*} (1\varepsilon ) \cong \varepsilon ) ,
s (\kappa ^{*} (\gamma ))
)
$$
(Compare with Example \ref{sf-zero}.)
\end{example}

\begin{example}\label{1-1-transfer}
The transfer homomorphism
$$
Imm^{fr(1)\times sf(0)} (1,1) \stackrel{(\ref{transfer1})}{\to } Imm^{fr(1)\times fr(0)} (1,1) 
\cong Imm^{fr(1)} (1,1) \cong \Z /2
$$
is computed
by the map
$$
(f,\kappa , \Xi ) \mapsto s (\kappa ^{*}(\gamma ));
$$
in particular, it is surjective.
\end{example}

\begin{example}
The diagram
$$
\begin{CD}
Imm^{fr(1)\times sf(0)} (1,1) @>>> Imm^{fr(1)} (1,1)  \\
@V \cong VV @V\cong VV \\
Imm^{fr(1)\times sf(4k)} (1,1+4k) @>>> Imm^{fr(1+4k)} (1,1+4k)  \\
@V \cong VV @V\cong VV \\
Imm^{sf(4k)} (1,4k) @>>> Imm^{fr(4k)} (1,4k)  
\end{CD}
$$
commutes
(here all horizontal arrows
are transfer homomorphisms (\ref{transfer1}),
right vertical arrows are suspension isomorphisms (\ref{suspension}),
left vertical arrows are isomorphisms from Corollary \ref{periodicity}.)

Thus Example
\ref{1-1-transfer}
implies surjectivity of the transfer homomorphism
$
Imm^{sf(4k)} (1,4k) \to Imm^{fr(4k)} (1,4k) 
$. 
\end{example}

\subsection{Dimension $2$}

Investigation of both immersion cobordism group
of $n$-dimensional manifolds
in $\R^{n+k}$ (without additional structures) 
and its oriented variant
comes back to papers
\cite{Rokhlin}, \cite{Wells}, \cite{Uchida}.
Here we denote these groups by 
$Imm^{O}(n,k)$ and $Imm^{SO}(n,k)$ correspondingly.
The diagram
$$
\begin{CD}
\Pi _n \cong Imm^{fr(k)} (n,k) @>>> Imm^{sf(k)} (n,k) \cong \Pi _{n+k} (\RP_k ^\infty  ) \\
@VVV @VVV \\
Imm^{SO} (n,k) @>>> Imm^{O} (n,k)
\end{CD}
$$
commutes.

For $k=1$
vertical arrows are isomorphisms
(see Statement \ref{reflection-2}).

In particular case, for
$k=1$ and $n=2$, 
results of \cite{Pinkall}, \cite{Hass-Hughes}
and Statement~\ref{reflection-all-sf}
imply the following.
Here $\text{Arf}$ 
denotes the isomorphism
given by virtue of the Arf-invariant \cite{Pinkall};
and $\text{Brown}$ stands for its generalization \cite{Brown1972}
(see also \cite{GM}, \cite{Ma}).

\begin{proposition}\label{surfaces}
1) The homomorphism
$$
\Z/2\stackrel{\text{Arf}}{\cong  } Imm^{fr(1)} (2,1) \stackrel{(\ref{t-skashivanie})}{\to }
Imm^{sf(1)} (2,1)\stackrel{\text{Brown}}{\cong } \Z /8
$$
has zero kernel.
\\
2)
The Boy surface $B:\RP^2 \looparrowright \R^3$
(with arbitrary skew-framing)
serves as a generator of
the group $Imm^{sf(1)} (2,1)\cong \Z /8 $.
The same holds true for the mirror image  $\bar B$ of
the Boy surface.
\\
3) Each immersion $\RP^2 \looparrowright \R^3$ 
is regularly homotopic either to
the Boy surface or to
its mirror image.
\end{proposition}

\subsection{Dimension $3$}

In dimension 3, we have already considered
Examples
\ref{simple-sphere}, 
\ref{replace-c-by-covering}.

\begin{example}\label{Freedman-sphere}
From \cite[Remark 1]{Freedman} it follows that
there exists a generic
immersion
${\mathbb F} :S^3 \looparrowright \R^4$ such that
the pair $(\mathbb F , \Xi _{\mathbb F})$,
where the isomorphism $\Xi _{\mathbb F} : \nu (\mathbb F) \cong \varepsilon $
is defined via external normal vector,
represents a generator
of the group $Imm^{fr(1)} (3,1) \cong \Pi _3 \cong \Z /24  $.
\end{example}

\begin{example}\label{gener-1}
The main theorem of the paper
\cite{Freedman}
implies that
the homomorphism
$$
Imm^{fr(1)} (3,1) 
\stackrel{(\ref{t-skashivanie}) }{\to } Imm^{sf(1)} (3,1) \cong \Pi _{4} (\RP^\infty _1)\cong \Z /2  
$$
is surjective.
Thus the triple
$({\mathbb F} , c=\const , \Xi _{\mathbb F}:
\nu (\mathbb F) \cong c^{*}(\gamma ) )$
represents a generator
of the group $Imm^{sf(1)} (3,1) $.
\end{example}

\begin{example}\label{Freedman-control}
The triple
$({\mathbb F} , c, \Xi _{\mathbb F} )$
considered as an element of the group
 $Imm^{sf(1),\RP^3} (3,1) $ differs from zero, because
 the natural homomorphism
$Imm^{sf(1),\RP^3} (3,1) 
\to Imm ^{sf(1)} (3,1)$\
(see formula (\ref{control}))
is surjective.

By Corollary \ref{order-2}
the triple
$({\mathbb F} , c, \Xi _{\mathbb F} )$
considered as an element of
the group $Imm^{sf(1),\RP^3} (3,1) $
has order $2$.
\end{example}

\section{Twisting and untwisting
of skew-framed immersions
with control in $\RP^3$}

Let
$
(f:M^n \looparrowright \R^{n+k} ,
\kappa :M^n \to \RP^3 , 
\Xi : \nu (f) \cong \kappa ^{*} (k\gamma ))$ be
a
$(k\gamma , \RP^3)$-immersion.

We will define two operations:
twisting,
that is,
transforming of a
$(k\gamma \oplus 4\varepsilon ,\RP^3)$-immersion
of an $n$-manifold
into a
$((k+4)\gamma ,\RP^3)$-immersion of the same manifold,
and untwisting,
that is, transforming
of a
$((k+4)\gamma ,\RP^3)$-immersion of an $n$-manifold
into a
$(k\gamma \oplus 4\varepsilon ,\RP^3)$-imemrsion
of the same manifold.

Points of the sphere
$S^3$ 
can be interpreted as quaternions of unit norm,
in particular, they can be multiplied.
According to notation in the introduction,
points of the total space
of $4\gamma $ over $\RP^3$
are written
as arrays of the form
$[x; \lambda _1,\lambda _2,\lambda _3,\lambda _4]$,
where $x\in S^3$, $\lambda _k\in \R$;
such array is a pair of identified arrays
$$
(x; \lambda_1,\lambda_2,\lambda_3,\lambda_4 ) \sim
(-x; -\lambda_1,-\lambda_2,-\lambda_3,-\lambda_4 ) .
$$
It will be convenient
to identify $\R^4$ with quaternions;
therefore to replace a 4-tuple of numbers
 $\lambda_1,\ldots , \lambda_4$
by a quaternion $w$, and to
write points of the total space of $4\gamma $ over $\RP^3$
as arrays
 $[x;w]$, where $x\in S^3$, $w\in \mathbb H$.

Using this notation, define an isomorphism
 $I_{\RP^3} : 4\varepsilon \cong 4\gamma $ over $\RP^3$ 
by the formula
\begin{equation}\label{I-RP3}
( [x] ; 
\lambda_1,\lambda_2,\lambda_3,\lambda_4) 
\mapsto
[x ; (\lambda_1 + \lambda_2 i + \lambda_3 j + \lambda_4 k ) \cdot x  ].
\end{equation}

\subsection{Twisting procedure $\mathbb T$}

Suppose we are given
a $(k\gamma \oplus 4\varepsilon ,\RP^3 )$-immersion
of an $n$-dimensional manifold:
$$(f: M ^n \looparrowright \R^{n+k+4} , \kappa : M\to \RP^3 , 
\Xi : \nu (f) \cong \kappa ^{*} (k\gamma \oplus 4\varepsilon ) ).$$
Denote  by
 $\Xi ^{\mathcal T}$ the composite isomorphism
$$ 
\nu (f) 
\stackrel{\Xi }{\cong }
\kappa ^{*}(k\gamma  \oplus 4\varepsilon )
\cong
\kappa ^{*}(k\gamma ) \oplus \kappa ^{*}( 4\varepsilon )
\stackrel{\id \oplus \kappa ^{*} (I_{\RP^3})}{\cong }
\kappa ^{*}(k\gamma ) \oplus \kappa ^{*}( 4\gamma )
= \kappa ^{*} ((k+4)\gamma ).
$$
[Symbol $\mathcal T$ is the first letter
of the word 
Twist.]

The triple
$$(f : M ^n \looparrowright \R^{n+k+4} , 
\kappa  :M\to \RP^3 , \Xi ^{\mathcal T} : 
\nu (f) \cong \kappa ^{*} ((k+4)\gamma ) )$$
is a $((k+4)\gamma  ,\RP^3)$-immersion.

Define the twisting
$\mathbb T$ by the formula
$$
\mathbb T (f,\kappa , \Xi ) = (f,\kappa , \Xi ^{\mathcal T}).
$$

\subsection{Untwisting procedure $\mathbb U$}

Now let a
$((k+4)\gamma ,\RP^3)$-immersion
of an $n$-manifold be given:
$$(f: M^n \looparrowright \R^{n+k+4} , \kappa :M\to \RP^3 , 
\Xi : \nu (f) \cong \kappa ^{*} ((k+4)\gamma ) ) .$$ 
Define an isomorphism $\Xi ^{\mathcal U}$ as the composition
$$ 
\nu (f) 
\stackrel{\Xi }{\cong }
 \kappa ^{*} ((k+4)\gamma )
 =\kappa ^{*} (k\gamma ) \oplus \kappa ^{*}(
 4\gamma )
\stackrel{\id \oplus \kappa ^{*} (I_{\RP^3 }^{-1}) }{\cong }
\kappa ^{*} (k\gamma ) \oplus \kappa ^{*}(
 4\varepsilon ) =
\kappa ^{*} (k\gamma \oplus 4\varepsilon ) .
$$
[Symbol $\mathcal U$ is the first letter of the word
Untwist.]

The triple
$$(f : M ^n\looparrowright \R^{n+k+4} , 
\kappa  : M \to \RP^3 , \mathcal U (\Xi ) : 
\nu (f) \cong \kappa ^{*} (k\gamma \oplus 4\varepsilon ) )
$$
is a $(k\gamma \oplus 4\varepsilon ,\RP^3 )$-immersion.

Define the
untwisting
$\mathbb U$ by the formula
$$
\mathbb U (f,\kappa , \Xi ) = (f,\kappa , \Xi ^{\mathcal U}).
$$

The following two statements are evident:

\begin{proposition}
For each $(k\gamma \oplus 4\varepsilon ,\RP^3 )$-immersion
$(f: M^n \looparrowright \R^{n+k+4}, \kappa : M\to \RP^3 ,  
\Xi  : \nu (f) \cong \kappa ^{*}(k\gamma \oplus 4\varepsilon ))$ 
we have
$$
\mathbb U \circ \mathbb T 
(f, \kappa ,  \Xi  )
=
(f, \kappa ,  \Xi  ) .
$$
For each $((k+4)\gamma  ,\RP^3 )$-immersion
$(f: M^n \looparrowright \R^{n+k+4}, \kappa : M\to \RP^3 ,  
\Xi  : \nu (f) \cong \kappa ^{*}((k+4)\gamma  ))$ 
we have
$$
\mathbb T \circ \mathbb U 
(f, \kappa ,  \Xi  )
=
(f, \kappa ,  \Xi  ) .
$$
\end{proposition}

\begin{proposition}
Procedures $\mathbb T$ and $\mathbb U$
induce mutually inverse isomorphisms
$$
Imm ^{sf(k)\times fr(4),\RP^3 } (n,k+4) 
\stackrel{\mathbb T, \mathbb U}{\longleftrightarrow }
Imm ^{sf(k+4),\RP^3 } (n,k+4) .
$$
\end{proposition}

We now obtain
from
(\ref{suspension}), (\ref{control})
the following known interpretation of the
James isomorphism
 \cite{James}. 
 It was given 
 and used by Mahowald \cite{Mahowald}.
We used it repeatedly in \cite{A-F}:

\begin{corollary}\label{periodicity}
For $n=0,1,2$
there are isomorphisms
$$
Imm ^{sf(k)} (n,k) 
\cong
Imm ^{sf(k+4)} (n,k+4) .
$$
\end{corollary}

There is no analogue of Corollary
\ref{periodicity}
for three-dimensional case. 
Moreover, twisting of cobordant
elements can be non-cobordant.

\begin{example}\label{E-1-zero}
Take the 4-suspension of the triple
 $(\mathbb E , c, \Phi )$
described in Example \ref{simple-sphere}.

This triple $E^4(\mathbb E , c, \Phi )$
represents zero element
of the group
$Imm^{sf(1)\times fr(4)}(3,5)$.

Since $c$ is a constant map,
the twisting
does not change the triple
$E^4(\mathbb E , c, \Phi )$.
Thus
$$
\mathbb T (E^4(\mathbb E , c, \Phi )) = 0
\quad 
\text{in the group}
\quad
Imm^{sf(5)} (3,5).
$$
\end{example}

\begin{example}\label{F-5}
Take the 4-suspension of the triple
$(\mathbb F , c ,\Xi _{\mathbb F})$
described in Example \ref{gener-1}.
This triple 
$E^4(\mathbb F , c ,\Xi _{\mathbb F})$
generates the group
$Imm^{sf(1)\times fr(4)} (3,5)
\cong Imm^{sf(1)} (3,1) \cong \Z /2$.
Since $c$ is a constant map,
the twisting does not change this triple.

Further,
the result of twisting
$$
E^4(\mathbb F , c ,\Xi _{\mathbb F})
\neq 0
\quad
\text{ in the group }
\quad
Imm^{sf(5)} (3,5).
$$

In fact, the triple 
$E^4(\mathbb F , c ,\Xi _{\mathbb F})$
is the image of a generator
$E^4(\mathbb F , \Xi _{\mathbb F})$ of the group
$Imm^{fr(5)}(3,5)\cong \Z /24$
under the homomorphism
$Imm^{fr(5)}(3,5)\to Imm^{sf(5)}(3,5)$, see
(\ref{t-skashivanie}).
This homomorphism is non-zero 
\cite[Theorem 1]{A-F}.

[Below we prove that $Imm^{sf(5)} (3,5) \cong \Z /2$,
see Theorem \ref{th1}.]
\end{example}

\begin{example}\label{rho-Twist}
In Example \ref{replace-c-by-covering}
we constructed the triple $(\mathbb E , \rho , \Xi _{\mathbb E})$.
Treated as an element of the group
$Imm^{sf(1)} (3,1)$, it equals zero.
Consider the image of this triple
under the composition of homomorphisms
\begin{equation}
\begin{split}
Imm^{sf(1),\RP^3} (3,1)
&
\stackrel{E^4}{\cong }
Imm^{sf(1)\times fr(4),\RP^3} (3,5)
\stackrel{\mathbb T}{\cong }
\\
&
\cong
Imm^{sf(5),\RP^3} (3,5)
\twoheadrightarrow
Imm^{sf(5)} (3,5) .
\end{split}
\end{equation}
The embedding $\mathbb E$ remains
unchanged.
Only the structure of normal bundle
changes.
It is easy to verify that
we will obtain, up to a sign,
the element 
 $t_1 (\nu )$,
where
$\nu $ is a generator
of the group
$Imm^{fr(5)}(3,5)\cong \Z /24$
(see \cite[Statements 4,5]{A-F}),
and the homomorphism 
$t_1 : Imm^{fr(5)} (3,5)
\to Imm^{sf(5)}(3,5)$
is defined by the formula (\ref{t-skashivanie}).
This homomorphism is non-zero
\cite[Theorem 1]{A-F}.
\end{example}

\begin{remark}\label{E-5}
Arguments of Example \ref{rho-Twist} also
show that the triple
 $(\mathbb E , \rho , \Xi _{\mathbb E})$
represents a non-zero element
in the group $Imm^{sf(1),\RP^3} (3,1)$.
\end{remark}

\section{Computation of the group $Imm^{sf(5)}(3,5)$}

\begin{theorem}\label{th1}
There are isomorphisms 
$$ 
Imm^{sf(5), \RP^3 }(3,5) \cong 
\Pi _8 (\RP^8 _5) \cong \pi _8 (\RP^8 _5)
\cong \Z/2 \oplus \Z /2, 
$$
$$
Imm^{sf(5)}(3,5) \cong 
\Pi _8 (\RP^\infty _5) \cong 
\pi _8 (\RP^\infty _5)\cong \Z/2 .
$$
\end{theorem}

\begin{proof}
First isomorphisms in each row
are well-known
\cite{Wells}.
Second arrows in each row
are isomorphisms by
dimensional reasons,
see Statement
\ref{emdedding-case}.

Let us prove that $Imm^{sf(1), \RP^3} (3,1) 
\cong \Z / 2 \oplus \Z /2 $.
Therefore the isomorphic group
$Imm^{sf(5), \RP^3 }(3,5)$ will also be
calculated.

Recall the skew-framed immersions
$ (\mathbb F , c , \Xi _{\mathbb F})$
and 
$
(\mathbb E , \rho , \Xi _{\mathbb E} )
$
introduced in Examples
\ref{gener-1} and \ref{simple-sphere}.

\begin{lemma}\label{3-kappa}
For arbitrary skew-framed immersion
of a 3-dimensional manifold
$(f_M:M^3 \looparrowright \R^4 , 
\kappa _M : M\to \RP^3 , 
\Xi _M: \nu (f_M) \cong \kappa _M^{*} (\gamma ))$
there exists a skew-framed cobordism
$(F_W:W\looparrowright \R^4\times I , 
\kappa  _W: W\to \RP^3 , 
\Xi _W: \nu (F_W) \cong \kappa _W^{*} (\gamma ))$
between
$(f_M,\kappa _M, \Xi _M)$
and a skew-framed immersion of the form
$$\delta\cdot (\mathbb F , c , \Xi _{\mathbb F})
\sqcup m\cdot 
(\mathbb E , \rho , \Xi _{\mathbb E} ),
$$
where $\delta \in \{ 0 ; 1\}$ and $m\geqslant 0$.
\end{lemma}

Let us underline that
here we strive for obtaining the inclusion
$
\kappa _W (W) \subset \RP^3 .
$

{\bf Proof of Lemma \ref{3-kappa}}

Since $Imm^{sf(1)} (3,1) \cong \Z /2  $,
there exists
a skew-famed cobordism
$(F_V:V\looparrowright \R^4\times I , 
\kappa  _V: V\to \RP^4 , 
\Xi _V: \nu (F_V) \cong \kappa _V^{*} (\gamma ))$
between 
$(f_M,\kappa _M, \Xi _M)$
and one of the two skew-framed immersions:
either
$(\mathbb F , c , \Xi _{\mathbb F})$
or
$(\mathbb E , c,\Phi )$.

Let us underline that
the image
$\kappa _V(V) \subset \RP^4$ 
in general
can not be moved by homotopy into
$\RP^3 \subset \RP^4$.

We change the cobordism
$V$; 
the new cobordism
 $W$ will satisfy the inclusion $\kappa _W(W)\subset \RP^3$;
 but doing this, we in general change the boundary:
 $\partial W \neq \partial V$.

We assume that the map $\kappa _V$ is smooth.
Let $pt \in \RP^4$ be a point choosed so that
it is a regular value for
$\kappa _V$ and $pt \notin \kappa _V (\partial V)$.
Then $\kappa _V ^{-1} (pt ) = \{ p_1 , \ldots , p_m\}$.
There exists an open disk neighborhood
$O(pt)$ 
of $pt\in \RP^4$ such that the
preimages
$O(p_k) = \kappa _V^{-1} (O(pt))$
are pairwise disjoint
disk neighborhoods
of the points $p_k$ for $k=1,\ldots , m$, 
and that $\kappa _V$ maps each of them diffeomorphically onto
$O(pt)$.

Replace the cobordism
$V$ by $W = V - \cup _{k=1}^m O(p_k)$.
Put
$f_W = {f_V }|_{W}$,
$\kappa _W = {\kappa _V} |_W$,
$\Xi _W = {\Xi _V} |_W$.
 
 Homotopying the map
 $\kappa _W : W \to \RP^4 - O(pt ) $,
 we can assume that
 the inclusion
$ \kappa _W ( W )\subset \RP^3 \subset  \RP^4 - O(pt ) $ holds.
(Moreover, the homotopy can be choosed
so that the map $\kappa _W$ 
is unchanged nearby $M\subset \partial W$.)
Reference to Example
\ref{replace-c-by-covering} and Remark \ref{perevorot-1}
finish the proof.
\qed

\smallskip

By Lemma \ref{3-kappa} the group
$Imm^{sf(1),\RP^3} (3,1)$
is generated by two elements:
$(\mathbb F , c , \Xi _{\mathbb F})$ and
$ (\mathbb E , \rho , \Xi _{\mathbb E} )$.

Each of them is non-zero
(see Example \ref{Freedman-control}, Remark
\ref{E-5}). 

The element
$(\mathbb F , c , \Xi _{\mathbb F})$
has order 2
(Corollary \ref{order-2}).
Using Statement \ref{reflection-all-sf},
one can show that
the triple
$(\mathbb E , \rho , \Xi _{\mathbb E})$
also has order 2 in the group
$Imm^{sf(1),\RP^3} (3,1)$.

Let us show that the triples
$(\mathbb E , \rho , \Xi _{\mathbb E})$ and
$(\mathbb F , c ,\Xi _{\mathbb F})$
represent different elements
in the group
$Imm^{sf(1),\RP^3} (3,1)$.
For this purpose, consider the epimorphism
$Imm^{sf(1),\RP^3} (3,1) 
\to Imm^{sf(1)} (3,1) \cong \Z /2$.
According to Example
\ref{replace-c-by-covering},
the first triple maps to zero.
By Example \ref{gener-1}, 
the second triple
maps to non-zero element.
Hence they are not cobordant in the group
$Imm^{sf(1),\RP^3} (3,1)$.

Finally, this group is abelian.
So, we have proved that 
$Imm^{sf(1), \RP^3} (3,1) \cong \Z / 2\oplus \Z /2$.
The first part of Theorem \ref{th1} is proved.

\medskip

Let us prove that $Imm^{sf(5)}(3,5)\cong \Z /2$.
By \cite[Theorem 1]{A-F},
the element
$t_1(\nu )$ 
of this group 
is non-zero;
here $\nu \in Imm^{fr(5)}(3,5) \cong \Pi _3 \cong \Z /24$ is a generator, and
$t_1 : Imm^{fr(5)}(3,5) \to Imm^{sf(5)}(3,5)$ is the homomorphism defined by the formula (\ref{t-skashivanie}).

Consider the epimorphism
$$
Imm^{sf(1),\RP^3} (3,1)
\cong 
Imm^{sf(5),\RP^3} (3,5)
\twoheadrightarrow
Imm^{sf(5)}(3,5).
$$
By Examples \ref{F-5}, \ref{rho-Twist},
this composition maps
the generators
$(\mathbb F , c ,\Xi _{\mathbb F})$
 and
$(\mathbb E , \rho , \Xi _{\mathbb E})$
of the group
$Imm^{sf(1),\RP^3} (3,1)$
into
$ t_1(\nu )$ up to a sign.
Hence the group $Imm^{sf(5)}(3,5)$ is cyclic,
and $t_1 (\nu )$ is its generator.
Since
$Imm^{sf(1),\RP^3} (3,1)\cong \Z /2 \oplus \Z /2$,
we obtain
$Imm^{sf(5)}(3,5)\cong \Z /2$.

The Theorem is proved.
\end{proof}

\section{Transformation Theorem}

Take a 5-skew-framed immersion of a 3-manifold
$(f: M^3 \looparrowright \R^8 ,\kappa : M\to \RP^3 , 
\Xi : \nu (f) \cong \kappa ^{*} (5\gamma ))$.

Let $\mathcal A : 5\gamma \to 5\gamma $ be an isomorphism over $\RP^3$.
Denote the isomorphism
$$
\nu (f) \stackrel{\Xi }{\cong } \kappa ^{*} (5\gamma )
\stackrel{\kappa ^{*}(\mathcal A ) }
{\cong } \kappa ^{*} (5\gamma ))
$$
by $\Xi ^{\mathcal A } $.

\begin{theorem}\label{zamena}
Suppose $\mathcal A : 5\gamma \to 5\gamma $ is
an isomorphism over $\RP^3$
that preserves orientation in each fiber.
Then 
for each 
5-skew-framed immersion of a 3-manifold
$(f: M^3 \looparrowright \R^8 ,\kappa : M\to \RP^3 , 
\Xi : \nu (f) \cong \kappa ^{*} (5\gamma ))$
the triples
$(f,\kappa , \Xi ^{\mathcal A } )$ 
and
$(f,\kappa , \Xi )$ are cobordant in the group
$Imm^{sf(5)} (3,5)$.
\end{theorem}

\begin{proof}
Lemma \ref{3-kappa} implies that
there exists a cobordism
$(F_W:W \looparrowright \R^8\times I,
\kappa _W : W\to \RP^3 , 
\Xi _W : \nu (F_W) \cong \kappa _W^{*} (5\gamma ))$
between the original triple
$(f,\kappa , \Xi )$ and a triple of the form
$$
\delta \cdot
\mathbb T (E^4(\mathbb F , c , \Xi _{\mathbb F }))
\sqcup
m\cdot \mathbb T (E^4(\mathbb E , \rho ,
\Xi _{\mathbb E })) ,
\quad
\text{ where }
\quad \delta \in \{ 0;1\},
\quad m\geqslant 0.
$$
We have
$
\mathbb T (E^4(\mathbb F , c , \Xi _{\mathbb F }))
=
E^4(\mathbb F , c , \Xi _{\mathbb F })$
(Example \ref{F-5}).
For convenience denote the triple
$\mathbb T (E^4(\mathbb E , \rho ,
\Xi _{\mathbb E }))$
by $(E^4 \mathbb E , \rho , \Psi _{\mathbb E})$.
Thus the triple
$(F_W,
\kappa _W , 
\Xi _W )$
gives a cobordism between
$(f,\kappa , \Xi )$ and a triple of the form
$$
\delta \cdot
E^4(\mathbb F , c , \Xi _{\mathbb F })
\sqcup
m\cdot (E^4(\mathbb E , \rho ,
\Psi _{\mathbb E }) ,
\quad
\text{ where }
\quad \delta \in \{ 0;1\},
\quad m\geqslant 0.
$$
Since $\kappa _W(W) \subset \RP^3$, 
the substitution
 $\mathcal A  $ can be done over the whole
 cobordism $W$.
 In other words, the triple
$(f_W,\kappa _W, \Xi _W^{\mathcal A  })$
is a cobordism between
$(f, \kappa , \Xi ^{\mathcal A  })$ 
and an immersion of the form
$$
\delta \cdot   
(E^4\mathbb F,c ,(E^4\Xi _{\mathbb F})^{\mathcal A  }) 
\sqcup
m\cdot (E^4 \mathbb E , \rho ,\Psi _{\mathbb E}^{\mathcal A  })
 ,
\quad
\text{where}
\quad
\delta \in \{ 0; 1 \},
\quad
m\geqslant 0.
$$

Idea of proof: show that the triples
$(E^4\mathbb F,c ,(E^4\Xi _{\mathbb F})^{\mathcal A  }) $
and
$(E^4\mathbb F,c ,E^4\Xi _{\mathbb F}) $
are cobordant,
and also the triples
$(E^4 \mathbb E , \rho ,\Psi _{\mathbb E}^{\mathcal A  })$
and 
$(E^4 \mathbb E , \rho ,\Psi _{\mathbb E})$
are cobordant.
Therefore, glueing
two cobordisms together,
we obtain
a cobordism between the triples
$(f,\kappa , \Xi ^{\mathcal A } )$ 
and
$(f,\kappa , \Xi )$.

The proof for the triple
$(E^4\mathbb F,c ,E^4\Xi _{\mathbb F}) $
is simple:
we can assume that the isomorphism
$\mathcal A$ is identical over
the point $c$ which is the image of the characteristic map, hence
$$
(E^4\mathbb F,c ,(E^4\Xi _{\mathbb F})^{\mathcal A }) 
=
(E^4\mathbb F,c ,E^4\Xi _{\mathbb F}) 
=
E^4(\mathbb F,c ,\Xi _{\mathbb F}) .
$$

The triple
$(E^4 \mathbb E , \rho ,\Psi _{\mathbb E}^{\mathcal A })$
needs a more detailed investigation.
This will be done in several steps.

{\bf Step 1.}
{\it 
Each automorphism
$\mathcal A: 5\gamma \to 5\gamma $ over $\RP^3$
that preserves orientation in fibers
is fiberwise homotopic to an automorphism of the form
$ {\mathcal B} \oplus \id _{\gamma }$,
where $\id : \gamma \to \gamma $ is the identity,
and
${\mathcal B} : 4\gamma \to 4\gamma $ preserves
orientation in fibers.}

In fact, it suffices to show that each map
 $\mu : \RP^3 \to SO(5)$ can be homotoped
 to a map
 whose image lies in $SO(4)$.
 For this purpose recall that
 there is a bundle $\pi : SO(5) \to S^4$ 
 with fiber $SO(4)$.
The composition $\pi \circ \mu : \RP^3 \to S^4$
is homotopic to a cell map, but the 
3-skeleton of $S^4$ is trivial,
so
 $\pi \circ \mu : \RP^3 \to S^4$
is homotopic to zero.
This homotopy
lifts to the total space $SO(5)$
and gives a homotopy
of the original map $\mu $ to a map
whose image is contained in a fiber, that is, in $SO(4)$.

So, we assume that 
$\mathcal A = 
{\mathcal B} \oplus \id _{\gamma }$,
where ${\mathcal B } : 4\gamma \to 4\gamma $ preserves orientation in fibers.

{\bf Step 2.}
Construct the isomorphism
$\Theta $ 
as a composition
$$
\nu (E^4\mathbb E) \stackrel{\Psi _{\mathbb E}}{\cong }
\rho ^{*}(5\gamma )
\stackrel{(\ref{P-pullback})}{\cong }
5\varepsilon ,
$$
and the isomorphism
$\Theta ^{\mathcal A}$ 
as a composition
$$
\nu (E^4\mathbb E) \stackrel{\Psi _{\mathbb E}^{\mathcal A}}{\cong }
\rho ^{*}(5\gamma )
\stackrel{(\ref{P-pullback})}{\cong }
5\varepsilon .
$$

{\bf 2.1.}
It follows from Statements \ref{kappa-homotopy}
\ref{reflection-all-sf}  that
the image of the pair
$(E^4\mathbb E , \Theta )$ 
under the homomorphism
$t_1 : Imm^{fr(5)} (3,5)\to Imm^{sf(5)}(3,5)$
(see formula (\ref{t-skashivanie}))
is cobordant to the triple
$(E^4\mathbb E , \rho , \Psi _{\mathbb E})$,
and the image of the pair
$(E^4\mathbb E , \Theta ^{\mathcal A})$ --- to the triple
$(E^4\mathbb E , \rho , \Psi _{\mathbb E}^{\mathcal A})$.

{\bf 2.2.}
In the framings $\Theta $, $\Theta ^{\mathcal A}$
the subbundle generated by vectors from second to fifth is transformed by twisting procedure.
It is convenient to assume
instead that
the subbundle generated by vectors from first to fourth is transformed.
More precisely,
define the isomorphism
$\Theta _1$
as the composition
$$
\nu (E^4\mathbb E) \cong \nu (\mathbb E) \oplus 4\varepsilon
=5\varepsilon = 4\varepsilon \oplus \varepsilon
\stackrel{I_{S^3}\oplus \id _{\varepsilon }}{\cong } 
4\varepsilon \oplus \varepsilon 
,
$$
and the isomorphism
$\Theta _1^{\tilde{\mathcal A}}$
as the composition
$$
\nu (E^4\mathbb E) \cong \nu (\mathbb E) \oplus 4\varepsilon
=5\varepsilon = 4\varepsilon \oplus \varepsilon
\stackrel{I_{S^3}\oplus \id _{\varepsilon }}{\cong } 
4\varepsilon \oplus \varepsilon 
\stackrel{\rho ^{*} (\mathcal B ) \oplus 
\id_{\varepsilon }}{\cong }
4\varepsilon \oplus \varepsilon .
$$
Here $I_{S^3} : 4\varepsilon \to 4\varepsilon $ 
is the automorphism
of the trivial bundle over $S^3$
defined by the formula
$$
(x; a_1,a_2,a_3,a_4)\mapsto
(x; (a_1+a_2i+a_3j+a_4k)\cdot x),
$$
compare with (\ref{I-RP3}).

Then 
the pairs
$(E^4\mathbb E , \Theta )$ � $(E^4\mathbb E , \Theta _1)$ are cobordant up to a sign,
and also are the pairs
$(E^4\mathbb E , \Theta ^{\tilde{\mathcal A}})$ and 
$(E^4\mathbb E , \Theta _1^{\tilde{\mathcal A}})$.

So, for the homomorphism
$t_1 : Imm^{fr(5)} (3,5)\to Imm^{sf(5)}(3,5)$
we have
$$
t_1 (E^4\mathbb E , \Theta _1)
\quad
\text{ is cobordant to }
\quad
(E^4\mathbb E , \rho , \Psi _{\mathbb E}),
$$
$$
t_1 (E^4\mathbb E , \Theta _1^{\tilde{\mathcal A}})
\quad
\text{ is cobordant to }
\quad
(E^4\mathbb E , \rho , 
\Psi _{\mathbb E}^{\tilde{\mathcal A}}).
$$

{\bf 2.3.}
$\Theta _1^{\tilde {\mathcal A} }
= 
\tilde {\mathcal A} \circ 
\Theta _1 $,
where
$\tilde {\mathcal A}$ is
the composition of the isomorphisms
$$
5\varepsilon 
\stackrel{(\ref{P-pullback})}{\cong }
\rho^{*}(5\gamma )
\stackrel{\rho ^{*} (\mathcal A )}{\cong }
\rho^{*}(5\gamma )
\stackrel{(\ref{P-pullback})}{\cong }
5\varepsilon 
$$
over $S^3$.

{\bf 2.4.}
$\tilde {\mathcal A}$
decomposes into a direct sum
$\tilde{\mathcal B} \oplus \id _{\varepsilon },$
where $\tilde {\mathcal B} :4\varepsilon \to 4\varepsilon $
is an isomorphism
which preserves orientation
in each fiber over
$S^3$.

{\bf 2.5.}
The isomorphism $\tilde {\mathcal B }:4\varepsilon \to 4\varepsilon $
possesses a symmetry property:
for the corresponding map
$\mu _{\tilde {\mathcal B}} : S^3 \to SO(4)$ we have
$$
\mu _{\tilde {\mathcal B}} (x) =
\mu _{\tilde {\mathcal B} }(-x ),
\quad x\in S^3.$$

{\bf Step 3.}
{\it Desuspension.}
Instead of the embedding
$E^4\mathbb E :S^3 \to \R^8$
consider the standard embedding $E^3\mathbb E :S^3 \to \R^7$.
We construct two framings of it.

Let us return to item {\bf 2.2} and denote
by $\Gamma $ the composite isomorphism
$$
\nu (E^3\mathbb E) \cong \nu (\mathbb E) \oplus 3\varepsilon
 = 4\varepsilon
\stackrel{I_{S^3}}{\cong } 
4\varepsilon  
,
$$
and by $\Gamma ^{\tilde {\mathcal B}}$ 
the composition 
$\tilde { \mathcal B} \circ \Gamma :
\nu (E^3\mathbb E) \cong 4\varepsilon $.

Then after suspension (\ref{suspension})
$$
E: Imm^{fr(4)} (3,4) \cong Imm^{fr(5)} (3,5)
$$
we obtain
$$
(E^3\mathbb E , \Gamma ) \mapsto (E^4\mathbb E , \Theta _1),
\quad
(E^3\mathbb E , \Gamma ^{\tilde{\mathcal B}}) 
\mapsto (E^4\mathbb E , \Theta _1^{\tilde {\mathcal A}}) .
$$

{\bf Step 4.}
Both stable Hopf invariant
$h^{st}: \Pi _3 \to \Z /2$
and homomorphism $t_1$ are surjective,
hence there exists
an isomorphism
$h^{sf}$
that makes the diagram
$$
\xymatrix{
 \Z /24  \cong \Pi _3 \cong Imm^{fr(5)} (3,5) \ar[rr]^{t _1} 
\ar[dr]^{h^{st}} & &
Imm^{sf(5)} (3,5)  \ar@{.>}[dl]_{h^{sf}}  \\
& \Z / 2 & 
}
$$
commutative.
So, for the elements
$a,b\in Imm^{fr(4)} (3,4)$
we have
$$t_1 (a) = t_1 (b)
\quad
\iff
\quad
h^{st}(a) = h^{st}(b).
$$

{\bf Step 5.}
The diagram
$$
\begin{CD}
Imm^{fr(4)} (3,4) @> E >> Imm^{fr(5)}(3,5) \\
@V h VV @V h ^{st} VV \\
\Z @> \mod 2 >> \Z /2
\end{CD}
$$
commutes.

Therefore it remains to show that
\begin{equation}\label{Hopf-diff}
h (E^3\mathbb E , \Gamma )
\equiv
h(E^3\mathbb E , \Gamma ^{\tilde{\mathcal B}})
\quad
\mod 2 .
\end{equation}

{\bf Step 6.}
It is known that
$(E^3\mathbb E , \Gamma )$ generates
the group
$Imm^{fr(4)} (3,4)$,
and its Hopf invariant equals 1.

Let us prove that
$h(E^3\mathbb E , \Gamma ^{\tilde{\mathcal B}})$
is odd.

The difference
$h(E^3\mathbb E , \Gamma )-
h(E^3\mathbb E , \Gamma ^{\tilde {\mathcal B}})$ 
coincides, up to a sign,
with degree of the map defined
by the first basis vector of the framing $\Gamma ^{\tilde {\mathcal B}}$, i.e. of the map
${\Delta }_{\tilde {\mathcal B}}:S^3 \to S^3$, 
where the image sphere is
the fiber
of spherization of the normal bundle
defined by $\Gamma $ \cite{pontr}.
From the symmetry property of the map
$\mu _{\tilde {\mathcal B}}$ 
(item {\bf 2.5})
follows the 
symmetry property of the map ${\Delta }_{\tilde {\mathcal B}}$, 
i.e.  ${\Delta }_{\tilde {\mathcal B}}(x) 
= {\Delta }_{\tilde {\mathcal B}}(-x)$ for all $x\in S^3$.
The symmetric map
${\Delta }_{\tilde {\mathcal B}}:S^3 \to S^3$
factorizes through the projection
map
$\rho :S^3\to \RP^3$.
Degree of this projection equals 2, hence
$\deg {\Delta }_{\tilde {\mathcal B}}$ is even.

Thus the  
congruence
(\ref{Hopf-diff}) is proved.

{\bf Step 7.}
The above arguments show that
the triples
$(E^4 \mathbb E , \rho , \Psi _{\mathbb E})$ and
$(E^4 \mathbb E , \rho , \Psi _{\mathbb E}^{\mathcal A })$
are cobordant in $Imm ^{sf(5)} (3,5)$.
As already said, this enables us
to glue together two cobordisms:
 $(F_W , \kappa _W , \Xi _W^{\mathcal A })$ 
and
$(F_W , \kappa _W , \Xi _W)$.
As a result, we obtain
a cobordism of the triples
$(f, \kappa , \Xi )$ and
$(f, \kappa , \Xi ^{\mathcal A })$.
Theorem is proved.
\end{proof}

\section{Product and Transfer Theorem}

The invariant
trivialization
(constructed with the help of Cayley
octonions)
$T\RP^7 \cong 7\varepsilon $
induces
by the Hirsch theorem
\cite[Theorem 6.3]{Hirsch-paper} 
an immersion
$f: \RP^7 \looparrowright \R^8$.
Fix an isomorphism
$\Xi : \nu (f) \cong \varepsilon $.
Also, take the identity map
$\kappa : \RP^7 \stackrel{\id }{\to } \RP^7 \subset \RP^\infty $.
The triple
$(f,\kappa , \Xi )$
represents an element of the group
$Imm^{sf(0)\times fr(1)}(7,1) $; denote it by~$\mathfrak p$.

Recall that
$$
Imm^{sf(5)\times fr(1)}(2,6) 
\cong
Imm^{sf(1)}(2,1)  \cong \Z /8
$$
and that the Boy surface 
$\RP^2 \looparrowright \R^3$
with arbitrary skew-framing can serve 
as a generator of the last group.
In \cite{A-F} we chose such a generator
and denoted it by $\upsilon \in Imm^{sf(1)} (2,1)$.
Its twisting
$\mathbb T(E^4(\upsilon ))$
is a generator of the group
$ Imm^{sf(5)} (2,5)$.

Define the homomorphism
\begin{equation}\label{skosn}
Imm^{sf(k)} (n,k) \to Imm^{sf(k+m)} (n-m,m+k) 
\end{equation}
(for $m=1$ see \cite{A-E}).
Suppose we are given 
 a representative of the first group, that is,
a triple
$
(f_M: M^n \looparrowright \R^{n+k},
\kappa  _M: M^n \to \RP^n ,
\Xi _M: \nu (f) \cong \kappa _M^{*}(k\gamma ))
$.
Assuming $\kappa _M $ to be transversal along $\RP^{n-m} \subset
\RP^n$,
put $N = \kappa _M^{-1} (\RP^{n-m})$.
It is a $(n-m)$-dimensional submanifold of $M$.
Besides, there is a natural isomorphism
$\nu (N,M) \cong \kappa _M^{*} (\nu (\RP^{n-m},\RP^n))
\cong \kappa _M^{*}(m\gamma ) $.
Let $f_N  :N\subset M \looparrowright \R^{n+k}$ be
the restriction of the immersion $f_M$;
$\kappa _N : N \subset M \to \RP^n $ the restriction of the map $\kappa _M$;
and let the isomorphism
$\Xi _N : \nu (f_N) \cong \kappa _N^{*}((k+m)\gamma )$
be obtained by the formula
$$
\nu (f_N) \cong \nu (N,M) \oplus \nu (f_M)
\cong \kappa _N  ^{*} (m\gamma ) \oplus 
\kappa _N^{*} (k\gamma ) =
\kappa _N^{*} ((m+k)\gamma ) .
$$
The correspondence
$$
(f_M,\kappa _M,\Xi _M)
\mapsto (f_N,\kappa _N, \Xi _N)$$
defines the homomorphism (\ref{skosn}).

\begin{proposition}
The image of the element $\mathfrak p $
under the homomorphism (\ref{skosn})
$$
Imm^{sf(0)\times fr(1)}(7,1) \to
Imm^{sf(5)\times fr(1)}(2,6) 
$$
generates the group $Imm^{sf(5)\times fr(1)}(2,6) 
\cong \Z /8 $. 
\end{proposition}

\begin{proof}
By definition, the image of the element
 $\mathfrak p$ 
 is represented by the triple
$(f_0,\kappa _0, \Xi _0)$,
where $f_0 : M_0 =\kappa ^{-1} (\RP^2) \subset \RP^7 \stackrel{f }{\looparrowright } \R^8$,
$\kappa _0 = \kappa |_{M_0} : M_0 \to \RP^\infty $,
$\Xi _0 = \Xi _{M_0} \oplus \kappa _0^{*} (\nu (\RP^2 , \RP^7 ))
: \nu (f_0) \cong \varepsilon \oplus 5\kappa _0^{*}\gamma $.

It is clear that $M_0 = \kappa ^{-1} (\RP^2) = \RP^2$.
So, the image of the element
 $\mathfrak p$ in the group $Imm^{sf(5)\times fr(1)} (2,6)$
is represented by an immersion of the
projective plane.
To show that it is cobordant
with 
$\pm E (\mathbb T (E^4 \upsilon ))$,
apply the inverse isomorphism
$\mathbb U \circ E^{-1} : 
Imm^{sf(5)\times fr(1)} (2,6)\cong Imm^{sf(1)} (2,1)$.
By construction, we are
to replace the skew-framing $\Xi _0 $ by
$$
\Psi : 
\nu (f_0) \stackrel{\Xi_0}{\cong }\varepsilon \oplus 5\kappa _0^{*}\gamma 
\stackrel{\id \oplus \kappa _0^{*}(I_{\RP^3}^{-1})}{\cong }
\varepsilon \oplus \kappa _0^{*}\gamma  \oplus 4\varepsilon ,
$$
and, in accordance with the Hirsch theorem,
replace
the immersion
$f_0 : \RP^2 \looparrowright \R^8$
by
a regularly homotopic immersion of the form
$g: \RP^2 \looparrowright \R^3 \subset \R^8$.

By items 3), 2) 
of Statement
\ref{surfaces}, the resulting element 
is cobordant to $\pm \upsilon $.
\end{proof}

\smallskip

To state Theorem \ref{th4},
we introduce elements
$c_1,c_2\in Imm^{sf(8)} (1,8)\cong \Z /2\oplus \Z/2$, they are represented 
correspondingly by the triples:
\begin{equation}\label{circles}
c_1 = (f_1 , \kappa _1 , \Xi _1 ),
\quad
c_2 = (f_2 , \kappa _2 , \Xi _2 ),
\end{equation}
where
 $f_1: S^1 \subset \R^2 \subset \R^9$ is the standard embedding,
$\kappa _1 : S^1 \to \RP^\infty $ is
homotopically non-trivial,
$\Xi _1 :  \nu (f_1) \cong \kappa _1^{*}(8\gamma )$ is an (arbitrary) isomorphism;
$f_2: S^1 \subset \R^2 \subset \R^9$ is the ``figure eight'' immersion,
$\kappa _2 : S^1 \to \RP^\infty $ is
homotopically non-trivial,
$\Xi _2 :  \nu (f_2) \cong \kappa _2^{*}(8\gamma )$ is an (arbitrary) isomorphism.

By $\eta \in Imm^{fr(1) } (1,1) \cong \Z /2$
the generator is denoted.

Here symbol $\times $ 
denotes the product homomorphism,
see, for example, \cite[(8)]{A-F}.

\begin{theorem}\label{th4}
1)
Under the
composition of the
homomorphisms 
$$
Imm^{sf(0)\times fr(1)}(7,1)  \otimes
Imm^{fr(1)}(1,1)  \stackrel{\times }{\to }
Imm^{sf(0)\times fr(2)}(8,2)
\stackrel{(\ref{skosn})}{\to } 
$$
$$
\to
Imm^{sf(5)\times fr(2)}(3,7) \stackrel{(E^2)^{-1}}{\cong  } Imm^{sf(5)}(3,5) 
$$
the image of the element
$\mathfrak p \otimes \eta $
is a generator of the group $Imm^{sf(5)}(3,5)\cong \Z /2$.

2)
Under the
composition of the
homomorphisms 
$$
Imm^{sf(0)\times fr(1)}(7,1) \otimes Imm^{sf(8)}(1,8) 
\stackrel{\times }{\to }  Imm^{sf(0)\times sf(8)}(8,8)
\stackrel{(\ref{transfer1})}{\to }
$$
$$
\to
Imm^{sf(0) \times fr(8)}(8,8)
\stackrel{(\ref{skosn})}{\to } 
Imm^{sf(5)\times fr (8)}(3,13) \stackrel{{(E^8)}^{-1}}{\cong }
Imm^{sf(5)}(3,5)
$$
both the image of the element
$\mathfrak p \otimes c_1$
and the image of the element
$\mathfrak p \otimes c_2$
are equal to the generator of the group
 $Imm^{sf(5)}(3,5)$.
\end{theorem}

\begin{proof}
Both assertions follow
from investigation 
of the same commutative diagram.
It is too large to be placed on a page,
and we will consider it in parts;
moreover, in this proof
groups $Imm$ are denoted shortly by $I$.

1)
We need to show that
going through the diagram down-right
$$
\begin{CD}
I^{sf(0)\times fr(1)}(7,1) \otimes I^{fr(1)}(1,1) @> (\ref{skosn})\times\id >>
I^{sf(5)\times fr(1)}(2,6) \otimes I^{fr(1)}(1,1)   \\
@V \times VV @V \times VV   \\
I^{sf(0)\times fr(2)}(8,2) @> (\ref{skosn}) >>  I^{sf(5)\times fr(2)}(3,7) 
\cong   I^{sf(5)}(3,5) 
\end{CD}
$$
the element
$\mathfrak p\otimes \eta $ 
maps to the generator.
Let us go right-down instead.
Going right, we obtain (up to a sign)
$E(\mathbb T (E^4(\upsilon ))) \otimes \eta $,
and going down this gives
the generator $E^2(t_1 (\nu ))$
\cite[Theorem 3]{A-F}.

2) The diagram 
$$
\begin{CD}
I^{sf(0)\times fr(1)}(7,1) \otimes I^{sf(8)}(1,8) @> \id\times (\ref{transfer1})  >>
I^{sf(0)\times fr(1)}(7,1) \otimes I^{fr(8)}(1,8)   \\
@V \times VV @V \times VV  \\
I^{sf(0)\times sf(8)}(8,8) @> (\ref{transfer1}) >>  I^{sf(0)\times fr(8)}(8,8)
\cong I^{sf(0)\times fr(2)} (8,2) 
\end{CD}
$$
commutes, and it
can be considered as an
extension to the right 
for the diagram of item 1).
Hence, in order to prove 2) it suffices
to check
that the transfer
$Imm^{sf(8)} (1,8)\to Imm^{fr(8)} (1,8)$
maps the elements $c_1$, $c_2$ to
the generator of the group $ Imm^{fr(8)} (1,8)\cong Imm^{fr(1)} (1,1)$.
This is verified directly.
\end{proof}

\section{Sketch of proof of the Theorem
on non-immersion of
$\RP^{10}$ into $\R^{15}$}

In \cite{A-F}, the authors
suggested and partly realised
a scheme of a geometric proof
of the Baum-Browder theorem
on non-immersion of $\RP^{10}$ into $\R^{15}$.
By the Sanderson lemma \cite[Lemma (9.7)]{BB}, 
it suffices to show that
$T\RP^{15}$ does not allow 10 
linearly independent
vector fields
over
$\RP^{10}\subset \RP^{15}$.
A stronger fact is valid: $T\RP^{15}$ does not allow 9 linearly independent vector fields
over $\RP^{10}\subset \RP^{15}$
\cite[Theorem (9.5)]{BB}.

Koschorke developed singularity approach
for the problem
of existence
of a monomorphism between two arbitrary vector
bundles over a manifold
\cite{K-book}. 
The method
prescribes to take
a general position morphism
$9\varepsilon \to
T\RP^{15} |_{\RP^{10}}$ over $\RP^{10}$ 
and study its
singular points --- the points where
this morphism has non-zero kernel.
It turns out that
the 
normal bundle of the
singular manifold
has special structure; 
Koschorke defines a special
cobordism group and shows that 
in this group
the singular manifold
is zero-cobordant if and only if
there exists
a monomorphism of the given bundles.

In our problem it is unconvenient to take
a morphism
$9\varepsilon \to T\RP^{15}$ over $\RP^{10}$.

Instead we will prove that there is no monomorphism
$10\varepsilon \to T\RP^{15}\oplus \varepsilon $ 
over $\RP^{10}$.
Secondly, we temporarily change the basis space ---
instead $\RP^{10}$ we will consider $\RP^{15}$
(the possibility to construct a convenient general position morphism arises
due to availability 
of the Hopf bundle $\RP^{15} \to S^8$).
The auxiliary problem
$10\varepsilon \to T\RP^{15}\oplus \varepsilon $ 
over $\RP^{15}$
does not satisfy conditions of the Koschorke theorem.
Nevertheless, we can consider
the singular manifold and its normal bundle.
The singularity happens
to be represented by an  8-manifold
with special structure of normal bundle;
this singularity manifold can be interpreted as
an element
$x \in Imm^{sf(0) \times fr(8)}(8,8)$, 
call it an auxiliary obstruction. 
Returning 
to the original problem
$10\varepsilon \to T\RP^{15}\oplus \varepsilon $ 
over $\RP^{10}$ corresponds to the homomorphism
of the 
``obstruction'' groups 
(here, inverted commas 
are used to mention that this word is not
fully exact for the problem over $\RP^{15}$);
this homomorphism decomposes as
$$
Imm^{sf(0) \times fr(8)}(8,8)
\to Imm^{sf(5)} (3,5)
\hookrightarrow
Imm^{sf(5)\times sf(6)} (3,11).
$$

In our next paper  we will show that in the group $Imm^{sf(0) \times fr(8)}(8,8)$
 holds the equation
$x = y + 2z + t, $
where $y$ is the image of the product
(see formula (\ref{circles}))
$\mathfrak p \otimes c_1 \in
Imm^{sf(0)\times fr(1)}(7,1) \otimes Imm^{sf(8)}(1,8)$
as in item 2) of Theorem $\ref{th4}$;
$t$ is an element which admits 
a clear description.
By Theorem $\ref{th1}$
the even element $2z$ maps to zero.
Usage of Theorem \ref{zamena}
will allow us to show that the element $t$
also maps to zero.
By  Theorem \ref{th4} and \cite[Theorem~3]{A-F}
the element $y$ maps to a non-zero
element.
Hence,
the image of the element $x$ also  differs from zero ---
but it is exactly the obstruction to existence
of 10 linearly independent vector fields
in the bundle $T\RP^{15} \oplus
\varepsilon $ over $\RP^{10}$.

\end{document}